 \documentclass[final,1p,times]{elsarticle}
\usepackage{amssymb}
\usepackage{comment}
\usepackage{relsize}
 \usepackage{amsthm}
\usepackage{amsmath,amssymb,amsopn,amsfonts,mathrsfs,amsbsy,amscd}

\def\br{[\;,\;]}

\newcommand{\Ric}{{\mathrm{Ric}}}
\newcommand{\ad}{\mathrm{ad}}
\newcommand{\img}{\mathrm{Im}}
\newcommand{\vect}{\mathrm{span}}
\newcommand{\prs}{\langle\;,\;\rangle}
\newcommand{\too}{\longrightarrow}
\newcommand{\om}{\omega}
\newcommand{\metric}{\langle\;,\;\rangle}
\newcommand{\g}{\mathfrak{g}}
\newcommand{\esp}{\quad\mbox{and}\quad}

\newcommand{\G}{{\mathfrak{g}}}

\newcommand{\h}{{\mathfrak{h}}}

\newcommand{\tr}{{\mathrm{tr}}}
\newcommand{\ric}{{\mathrm{ric}}}
\newcommand{\Ri}{{\mathrm{Ric}}}

\newcommand{\B}{{\cal B}}

\newcommand{\di}{\displaystyle}

\newcommand{\al}{\alpha}

\newcommand{\e}{\epsilon}

\newcommand{\la}{\lambda}

\newcommand{\de}{\delta}

\font\bb=msbm10

\def\B{\hbox{\bb B}}
\def\R{\hbox{\bb R}}

\def\N{\hbox{\bb N}}

\def\C{\hbox{\bb C}}
\newtheorem{Def}{Definition}[section]
\newtheorem{theo}{Theorem}[section]
\newtheorem{pr}{Proposition}[section]
\newtheorem{Le}{Lemma}[section]
\newtheorem{co}{Corollary}[section]

\newtheorem{exem}{Example}

\newtheoremstyle{named}{}{}{\itshape}{}{\bfseries}{.}{.5em}{\thmnote{#3 }}
\theoremstyle{named}

\begin{document}

\begin{frontmatter}


 %\title{Left invariant para-K\"ahler and hyper-para-K\"ahler structures on Lie groups\tnoteref{label1}}
 %\author{\corref{cor1}\fnref{label2}}
 
 %\fntext[label3]{This research was conducted within the framework of Action concert\'ee CNRST-CNRS Project SPM04/13.}

\title{  Classification of Einstein Lorentzian 3-nilpotent Lie groups with 1-dimensional  nondegenerate center}

 %\author[label1,label2,label3]{}
 \author[label1,label2]{ Mohamed Boucetta, Oumaima Tibssirte}
 \address[label1]{Universit\'e Cadi-Ayyad\\
 	Facult\'e des sciences et techniques\\
 	BP 549 Marrakech Maroc\\e-mail: m.boucetta@uca.ma
 }
 
 \address[label2]{Universit\'e Cadi-Ayyad\\
 	Facult\'e des sciences et techniques\\
 	BP 549 Marrakech Maroc\\e-mail: oumayma1tibssirte@gmail.com 
 }
% \address[label3]{}

%\author{}

%\address{}

\begin{abstract} We give a  complete classification of Einstein Lorentzian 3-nilpotent simply connected Lie groups with 1-dimensional  nondegenerate center.
\end{abstract}


\end{frontmatter}

% \linenumbers




%% References


\section{Introduction} \label{section1}
The study of  left-invariant Einstein Riemannian metrics on Lie groups is a research area that had made huge progress in the last decades (see \cite{heber, lauret, lauret1}). However, the indefinite case remains unexplored in comparison and only few significant results had been published in this matter with many questions that are still open (see \cite{guediri, Dconti1, bou0}).\\
In \cite{bou0}, the authors began an inspection of Einstein Lorentzian nilpotent Lie algebras  following guidelines from previous studies of the $2$-step nilpotent case (see \cite{bou1} and \cite{guediri}).
The main Theorem of \cite{bou0} states that Einstein nilpotent Lie algebras with degenerate center are exactly Ricci-flat  and are obtained by a double  extension process starting from a Euclidean vector space (see \cite[Theorem $4.1$]{bou0} and \cite{medina} for the original definition of the double extension). This class of Lie algebras includes all Einstein Lorentzian nilpotent Lie algebras that are either $2$-step or of dimension less than $5$, in fact as a concrete application of the main Theorem, the authors were able to give a full classification of the latter.\\ Dimension $6$ however falls outside the context of this result as the authors presented the first example in this situation of an Einstein nilpotent Lie algebra with non-degenerate center, which also happens to be $3$-step nilpotent. Einstein nilpotent Lie algebras that are non Ricci-flat has been shown to exist in the Lorentzian setting (see \cite{Dconti1}) and according to \cite[Theorem $4.1$]{bou0} these must have non-degenerate center as well. So the study of Einstein Lorentzian nilpotent Lie algebras with nondegenerate center becomes a natural and challenging problem and
the present paper can be seen as a first attempt to find a general pattern for these Lie algebras. We start by  the $3$-step nilpotent case and we develop a new approach which can be used later in the general case. Let us give a brief summary of our method and state our main result. \\
Let $(\h,[\;,\;])$ be a  $k$-nilpotent Lie algebra and $\metric$ an Einstein Lorentzian metric on $\h$ such that the center of $\h$ is non-degenerate. Then $Z(\h)$ is non degenerate Euclidean (see \cite{bou0}) and, naturally, we get the orthogonal spitting 
$$\h=\mathrm{Z}(\h)\overset\perp\oplus\g.$$
The Lie bracket on $\h$ splits  accordingly as $[u,v]=\omega(u,v)+[u,v]_0$ for any $u,v\in\g$, where $[\;,\;]_0$ is a Lie bracket on $\g$ and $\omega:\g\times\g\longrightarrow\mathrm{Z}(\h)$ is a $2$-cocycle of $(\G,[\;,\;]_0)$. It turns out that $(\g,[\;,\;]_0,\metric_{|\g\times\g})$ is a Lorentzian $(k-1)$-nilpotent Lie algebra and the Einstein equation on $\h$ can be expressed entirely by means of the Lie algebra $\g$ as a sort of compatibility condition between $\omega$ and the Ricci curvature $\Ric_\g$ of $(\G,\prs_\G,[\;,\;]_0)$ (see Proposition \ref{prrici}). This shift in perspective is especially useful when the Lie algebra $\h$ is $3$-step nilpotent since $\g$ is 2-nilpotent and, for instance, we can  show that every Einstein Lorentzian $3$-step nilpotent Lie algebra with non-degenerate center has positive scalar curvature (Theorem \ref{main1}). It also gives rise to the notion of $\omega$-quasi Einstein Lie algebras (see Definition \ref{DefOmegaQuasi}). A careful study  of $\omega$-quasi Einstein 2-nilpotent Lie algebras leads to our main result, namely the classification of Einstein Lorentzian $3$-step nilpotent Lie algebras with $1$-dimensional non-degenerate center. Surprisingly enough, these are shown to only exist in dimensions $6$ and $7$.
\begin{theo}
\label{Thprin1}
Let $\h$ be a $3$-step nilpotent Lie algebra with $\dim\mathrm{Z}(\h)=1$. Let $\metric$ be a Lorentzian metric on $\h$ such that $\mathrm{Z}(\h)$ is non-degenerate, then $\metric$ is Einstein if and only if $\h$ is Ricci-flat and has one of the following forms :\\

$(i)$ $\dim\h=6$ and $\h$ is isomorphic to $\mathrm{L}_{6,19}(-1)$, i.e., $\h$ has a basis $(f_i)_{i=1}^6$ such that the non vanishing Lie brackets are
\[ [f_1,f_2]=f_4,[f_1,f_3]=f_5,[f_2,f_4]=f_6,[f_3,f_5]=-f_6 \]
and the metric  is given by :
		\begin{equation}
	\label{6dimmetric1}
	\metric:=f_1^*\otimes f_1^*+2f_2^*\otimes f_2^*+2f_3^*\otimes f_3^*+4\alpha^4f_6^*\otimes f_6^*-2\alpha^2 f_4^*\odot f_5^*,\;\;\;\alpha\neq 0.
	\end{equation}
	
	$(ii)$ $\dim \h=7$ and $\h$ is isomorphic to the nilpotent Lie algebras $147E$ found in the classification given in \cite{gong1998}(p. $57$). In precise terms, there exists a basis $\{f_i\}_{i=1}^7$ of $\h$ where the non vanishing Lie bracket are given  by :
	\begin{equation}
	\label{structdim7}
	[f_1,f_2]=f_5,\;[f_1,f_3]=f_6,\;[f_2,f_3]=f_4,\;[f_6,f_2]=(1-r)f_7,\;[f_5,f_3]=-rf_7,\;[f_4,f_1]=f_7,
	\end{equation}
	with $0<r<1$,  and the metric has the form:
	\begin{equation}
	\label{structmetric7}
	\langle\;,\;\rangle=f_1^*\otimes f_1^*+f_2^*\otimes f_2^*+f_3^*\otimes f_3^*-a f_4^*\otimes f_4^*+ar f_5^*\otimes f_5^*+a(1-r)f_6^*\otimes f_6^*+a^2 f_7^*\otimes f_7^*,\;\;\;a>0.
	\end{equation}
\end{theo}
\paragraph{Outline of the paper}  In Section $2$ we give some preliminaries on Pseudo-Riemannian Lie algebras as well as all the notations needed for subsequent development. In Section $3$, we describe an Einstein Lorentzian nilpotent Lie algebra $\h$ with non-degenerate center by means of its center, a nilpotent Lorentzian Lie algebra $\g$ of lower order, and a $2$-cocycle $\omega\in\mathrm{Z}^2(\g,\mathrm{Z}(\h))$, these are called \emph{the attributes} of $\h$ (see Definition \ref{DefAttri}). The main result of this section is Theorem \ref{main1} in which we prove that any Einstein Lorentzian $3$-step nilpotent Lie algebra of non-degenerate center has positive scalar curvature, at the end of the section we introduce the notion of \emph{$\omega$-quasi Einstein Lie algebra}. The remainder of the document is devoted for the proof of the central results. As the reader can see, the proof of Theorem \ref{Thprin1}  turns out to be difficult  and it is based on a sequence of Lemmas (Lemma \ref{genlem0}, \ref{genlem1} and \ref{Lemmaimp}). This suggests that the complete study of Einstein Lorentzian nilpotent Lie algebras with nondegenerate center is a challenging mathematical problem.
\section{Preliminaries}\label{section2}
A \emph{pseudo-Euclidean  vector space } is  a real vector space  of finite dimension $n$
endowed with  a
nondegenerate symmetric inner product  of signature $(q,n-q)=(-\ldots-,+\ldots+)$.  When
the
signature is $(0,n)$
(resp. $(1,n-1)$) the space is called \emph{Euclidean} (resp. \emph{Lorentzian}).

Let $(V,\prs)$ be a pseudo-Euclidean vector space of signature  $(q,n-q)$. A vector $u\in V$  is called \emph{spacelike} if $\langle u,u\rangle>0$, \emph{timelike} if $\langle u,u\rangle<0$ and \emph{isotropic} if  $\langle u,u\rangle=0$.
A family
$(u_1,\ldots,u_s)$ of vectors in $V$ is called \emph{orthogonal}  if, for $i,j=1,\ldots,s$
and $i\not=j$, $\langle u_i,u_j\rangle=0$. An orthonormal basis of $V$ is an orthogonal basis $(e_1,\ldots,e_n)$ such that $\langle e_i,e_i\rangle=\pm1$.
For any endomorphism $F:V\too V$, we denote by $F^*:V\too V$ its adjoint with respect to $\prs$.

It is a well-known fact that the study of the curvature of left invariant pseudo-Riemannian metrics on Lie groups reduces to the study of its restriction to their Lie algebras. Let us recall some definitions and fix some notations. The reader can consult \cite{bouc1} or \cite{bou0} for details.

Let $({\h},[\;,\;],\prs)$ be a {\it pseudo-Euclidean Lie algebra}, i.e, a Lie algebra endowed with a pseudo-Euclidean product.
The \emph{ Levi-Civita product} of ${\h}$ is the bilinear map $\mathrm{L}:{\h}\times{\h}\too{\h}$ given by  Koszul's
formula
\begin{eqnarray}\label{levicivita}2\langle
\mathrm{L}_uv,w\rangle&=&\langle[u,v],w\rangle+\langle[w,u],v\rangle+
\langle[w,v],u\rangle.\end{eqnarray}
For any $u,v\in{\h}$, $\mathrm{L}_{u}:{\h}\too{\h}$ is skew-symmetric and $[u,v]=\mathrm{L}_{u}v-\mathrm{L}_{v}u$.
The curvature of ${\h}$ is given by
$$
\label{curvature}K(u,v)=\mathrm{L}_{[u,v]}-[\mathrm{L}_{u},\mathrm{L}_{v}].
$$
The Ricci curvature $\mathrm{ric}:{\h}\times{\h}\too\R$ and its Ricci operator $\Ri:{\h}\too{\h}$ are defined by $$\langle \Ri (u),v\rangle=\mathrm{ric}(u,v)=\mathrm{tr}\left(w\too
K(u,w)v\right).$$ 
A pseudo-Euclidean Lie algebra is called \emph{flat} (resp. \emph{Ricci-flat}) if $K=0$
(resp. $\ric=0$). It is called $\la$-Einstein if there exists a constant $\la\in\R$ such that $\Ri=\la\mathrm{Id}_{\h}$.

In this paper, we deal with nilpotent Lie algebras and in this case the ricci curvature is given by
\begin{equation}\label{ricci1}
\ric(u,v)=-\frac12\tr(\ad_u\circ
\ad_v^*)-\frac14\tr(J_u\circ J_v),\end{equation}where $J_u$ is the skew-symmetric endomorphism given by $J_u(v)=\ad_v^*u$. Moreover,
if ${\mathcal J}_1$
and
${\mathcal J}_2$ denote the symmetric endomorphisms given by
\begin{equation}
\label{ops3}
\langle{\mathcal J}_1 u,v\rangle=\tr(\ad_u\circ\ad_v^*),\;
\langle{\mathcal J}_2 u,v\rangle=-\tr(J_u\circ J_v)=\tr(J_u\circ J_v^*).\end{equation}
then the Ricci operator has the following expression
\begin{equation}\label{riccinilpotent}\mathrm{Ric}=-\frac12{\mathcal
	J}_1+\frac14{\mathcal J}_2,\end{equation}
The endomorphisms  ${\mathcal J}_1$ and ${\mathcal J}_2$ can be expressed in a useful way. Indeed,
if
$(e_1,\ldots,e_n)$ is  a basis of
$[\h,\h]$, then, for any $u,v\in{\h}$, the Lie bracket can be written
\begin{equation}\label{bracket}[u,v]=\sum_{i=1}^n\langle J_iu,v\rangle
e_i,\end{equation}where $(J_1,\ldots,J_n)$ is a family of  skew-symmetric
endomorphisms with respect to $\prs$. 
This family  will be called
\emph{Lie  structure endomorphisms} associated to $(e_1,\ldots,e_n)$. 
The following proposition   will be very useful later. See \cite[Proposition 2.3]{bou0} for its proof.
\begin{pr}Let $({\h},\prs)$ be a pseudo-Euclidean   Lie algebra,
	$(e_1,\ldots,e_n)$  a basis of
	$[\h,\h]$ and $(J_1,\ldots,J_n)$ the corresponding structure endomorphisms. Then 
	\begin{equation}\label{invariant}
	{\mathcal J}_1=-\sum_{i,j=1}^n\langle e_i,e_j\rangle J_i\circ J_j\quad
	\mbox{and}\quad {\mathcal J}_2u=-\sum_{i,j=1}^n\langle e_i,u\rangle{\tr}(J_i\circ
	J_j)e_j.\end{equation}In particular, $\tr{\mathcal J}_1=\tr{\mathcal J}_2$.

\end{pr}

\section{Lorentzian nilpotent Einstein Lie algebras with nondegenerate center}\label{section3}

In \cite{bou0}, we  studied Lorentzian nilpotent Einstein Lie algebras with degenerate center and we gave the first example of a Lorentzian 3-step nilpotent Ricci-flat Lie algebra with nondegenerate center. We also showed that an Einstein Lorentzian nilpotent Lie algebra with non zero scalar curvature must have a nondegenerate center. A first example of such algebras was given in \cite{Dconti1}. A 2-step nilpotent Einstein Lorentzian Lie algebra must be Ricci-flat with degenerate center so it is natural to start by studying 3-step nilpotent Einstein Lorentzian Lie algebras with nondegenerate center which, according to  \cite[Corollary 3.1]{bou0},  must be Euclidean.

Any nilpotent Lie algebra can be obtained by Skjelbred-Sund's method, namely, by an extension from a nilpotent Lie algebra of lower dimension and a 2-cocycle with values in a vector space (see \cite{graaf}). We will adapt this method to our study.

Let $(\h,\prs_\h)$ be a Lorentzian $k$-step nilpotent Lie algebra of dimension $n$ with nondegenerate  Euclidean center $Z(\h)$ of dimension $p\geq1$.  Denote by $\prs_z$ the restriction of $\prs$ to $Z(\h)$, $\g=Z(\h)^\perp$ and by $\prs_\g$ the restriction of $\prs$ to $\g$. We get that
\[ \h=\g\stackrel{\perp}{\oplus} Z(\h), \]
where $(Z(\h),\prs_z)$ is an Euclidean vector space  and $(\g,\prs_\g)$ is a Lorentzian vector space. Moreover, for any $u,v\in\g$, we have
\[ [u,v]=[u,v]_\g+\om(u,v), \]where $[u,v]_\g\in\G$ and $\om(u,v)\in Z(\h)$.  The Jacobi identity applied to $\br$ is easily seen equivalent to
$(\g,\br_\g)$ being a Lie algebra and $\om:\G\times\G\too Z(\h)$ a 2-cocycle of $\g$ with respect to the trivial representation of $\g$ in $Z(\h)$, namely, for any $u,v,w\in\g$,
\[ \om([u,v]_\G,w)+\om([v,w]_\G,u)+\om([w,u]_\G,v)=0.\]
Moreover,  
\begin{equation}\label{rigid}Z(\g)\cap\ker\om=\{0\}\esp C^n(\h):=[C^{n-1}(\h),\h]=C^{n}(\G)+\om(C^{n-1}(\G),\G), \end{equation}for any $n\in\N$.
This implies that $(\h,\br)$ is $k$-step nilpotent if and only if $(\g,\br_\g)$ is $k-1$-step nilpotent and $C^{k-2}(\G)\nsubset\ker\om$.

\begin{Def} \label{DefAttri}Let $(\h,\br,\prs_\h)$ be a  Lorentzian nilpotent Lie algebra with nondegenerate Euclidean center. We call the triple $(\g,\prs_\g,\br_\g)$, $(Z(\h),\prs_z)$ and $\om\in Z^2(\g,Z(\h))$ the attributes of $(\h,\br,\prs_\h)$.

\end{Def}

 We proceed now to express the Ricci curvature of $\h$ in terms of its attributes	$(\g,\prs_\g,\br_\g)$, $(Z(\h),\prs_z)$ and $\om\in Z^2(\g,Z(\h))$. For any $u\in \G$, we consider $\om_u:\G\too Z(\h)$, $v\too\om(u,v)$, $\om_u^*:Z(\h)\too\G$ its transpose given by
\[ \langle\om_u^*(x),v\rangle_\g=\langle \om(u,v),x\rangle_z. \]
For any $x\in Z(\h)$, we define $S_x:\G\too\G$  by
\[ S_x(u)=\om_u^*(x). \]It is clear that $S_x$ is skew-symmetric. Recall that, for any $u\in\g$, we denote by $J_u:\G\too\G$ the skew-symmetric endomorphism given by $J_u(v)=\ad_v^*(u)$.

On the other hand, define the  endomorphism $D:\G\too\G$ by
\begin{equation}\label{D}
\langle Du,v\rangle_\G=\tr(\om_u^*\circ\om_v).
\end{equation}It is clear that $D$ is symmetric with respect to $\prs_\G$. Let $(z_1,\ldots,z_p)$ be a basis of $Z(\h)$. There exists a unique family $(S_1,\ldots,S_p)$ of skew-symmetric endomorphisms such that, for any $u,v\in\G$,
\begin{equation}\label{car} \om(u,v)=\sum_{i=1}^p\langle S_iu,v\rangle_\G z_i.
\end{equation}  This family  will be called
\emph{$\omega$-structure endomorphisms} associated to $(z_1,\ldots,z_p)$.
A direct computation using \eqref{D} and \eqref{car} shows that
\begin{equation}\label{Df}
D=-\sum_{i,j}\langle z_i,z_j\rangle_z S_i\circ S_j.
\end{equation}

This operator has an interesting property.
\begin{pr}\label{ompr} If $\om$ satisfies 
	\begin{equation}\label{om}\om(\ad_u^*v,w)+\om(v,\ad_u^*w)=0\end{equation}for any $u,v,w\in\G$, then
	$D$ is a derivation of $(\G,\br_\G)$.
	
\end{pr}

\begin{proof} Since $\om$ is a 2-cocycle then
	\[ \om_{[u,v]_\G}=\om_u\circ\ad_v-\om_v\circ\ad_u. \]
	We also have, for any $u,v,w\in\G$
	\begin{align*}
	\langle [Du,v]_\G,w\rangle_\G+\langle [u,Dv]_\G,w\rangle_\G&=-\tr(\om_{\ad_v^*w}\circ\om_u^*)+
	\tr(\om_{\ad_u^*w}\circ\om_v^*),\\
	&=\tr(\om_w\circ\ad_v^*\circ\om_u^*)-\tr(\om_w\circ\ad_u^*\circ\om_v^*),\\
	\langle D[u,v]_\G,w\rangle&=\tr(\om_{[u,v]_\G}\circ\om_w^*)\\
	&=\tr(\om_u\circ\ad_v\circ\om_w^*)-\tr(\om_v\circ\ad_u\circ\om_w^*). 
	\end{align*}
\end{proof}

\begin{pr} \label{prrici}The Ricci curvature $\ric_\h$ of $(\h,\br,\prs_\h)$ is given by
	\begin{align*}
	\ric_\h(u,v)&=\ric_\G(u,v)
	-\frac12\tr(\om_u^*\circ\om_v),\quad u,v\in\G,\\
	\ric_\h(x,y)&=-\frac14\tr(S_x\circ S_y),\quad x,y\in Z(\h),\\
	\ric_\h(u,x)&=-\frac14\tr(J_u\circ S_x),\quad x\in Z(\h),u\in\G,
	\end{align*}where $\ric_\G$ is the Ricci curvature of $(\G,\br_\g,\prs_\g)$.

\end{pr}
\begin{proof} According to \eqref{ricci1}, for any $a,b\in\h$,
	\[ \ric_\h(a,b)=-\frac12\tr(\ad_a^\h\circ(\ad_b^\h)^*)-\frac14\tr(J_a^\h\circ J_b^h ), \]where $\ad_a^h:\h\too\h$, $b\mapsto [a,b]$ and $J_a^\h:\h\too\h$, $b\mapsto (\ad_b^\h)^*(a)$. The desired formula will be a consequence of this one and the following relations. For any $u\in\G$, $x\in Z(\h)$, with respect to the splitting $\h=\G\oplus Z(\h)$, we have
	\[ \ad_u^\h=\left(\begin{array}{cc} \ad_u^\G&0\\\om_u&0   \end{array}   \right),\; 
	J_u^\h=\left(\begin{array}{cc}  J_u^\G&0\\0&0   \end{array}   \right),\; J_x^\h=\left(\begin{array}{cc}  S_x&0\\0&0   \end{array}   \right)\esp \ad_x^\h=0.  \]

\end{proof}

\begin{co}\label{coe} $(\h,\br,\prs_\h)$ is $\la$-Einstein if and only if  for any $u,v\in\G$ and $x,y\in Z(\h)$,
	\begin{equation} \label{es} \ric_\g(u,v)=\la\langle u,v\rangle_\g+\frac12\tr(\om_u^*\circ\om_v),\;
	\tr( J_u\circ S_x)=0\esp \tr(S_x\circ S_y)=-4\la \langle x,y\rangle_z.
	\end{equation}
	
\end{co}
Let us derive some consequences of Proposition \ref{prrici} and Corollary \ref{coe}. In what follows $\h$ will be an Einstein Lorentzian nilpotent Lie algebra with nondegenerate center, we denote $\br_\h$ its Lie bracket, $\prs_\h$ its Lorentzian product and $(\G,\br_\G,\prs_\G)$, $(Z(\h),\prs_z)$ and $\om\in Z^2(\G,Z(\h))$ its attributes.

Recall that a pseudo-Euclidean Lie algebra $(\G,\br,\prs)$ is called Ricci-soliton if there exists a constant $\la\in\R$ and derivation $D$ of $\G$ such that $\Ric_\G=\la\mathrm{Id}_\G+D$. By combining Corollary \ref{coe} and Proposition \ref{ompr} we get the following result.

\begin{pr}\label{sol} Let $\h$ be a Einstein Lorentzian nilpotent Lie algebra with Euclidean nondegenerate center. If $\om$ satisfies \eqref{om} then $(\G,\br_\G,\prs_\G)$ is Ricci-soliton. 
	
\end{pr}
\begin{pr}
	Let $\h$ be a $\la$-Einstein Lorentzian nilpotent Lie algebra with non-degenerate center. If $\lambda\neq 0$ then the cohomology class of  the attribute $\omega$ is non trivial. In particular, $H^2(\G,Z(\g))\not=\{0\}$.
\end{pr}
\begin{proof}
	Suppose that there exists $\al\in\G$ such that, for any $u,v\in\G$, $ \omega(u,v)=-\alpha([u,v]_\G)$. Fix an orthonormal basis $\{e_1,\dots,e_n\}$ of $\g$ with $\langle e_1,e_1\rangle=-1$. For any $x\in Z(\h)$, we have :
	\begin{align*}
	\tr(S_x^2) &=\langle S_x(e_1),S_x(e_1)\rangle_\g -\sum_{i=2}^n\langle S_x(e_i),S_x(e_i)\rangle_\g \\
	&=\langle \omega_{e_1}^*(x),S_x(e_1)\rangle_\g-\sum_{i=2}^n \langle \omega_{e_i}^*(x),S_x(e_i)\rangle_\g\\
	&=-\langle \ad_{e_1}^*\circ\alpha^*(x),S_x(e_1)\rangle_\g+\sum_{i=2}^n\langle \ad_{e_i}^*\circ\alpha^*(x),S_x(e_i)\rangle_\G\\
	&=-\langle J_{\alpha^*(x)}(e_1),S_x(e_1)\rangle_\g+\sum_{i=2}^n \langle J_{\alpha^*(x)}(e_i),S_x(e_i)\rangle_\g\\
	&=-\tr(J_{\alpha^*(x)}\circ S_x).
	\end{align*}
	By virtue of Corollary \ref{coe},  we get that $\lambda\langle x,x\rangle_z=0$ for any $x\in Z(\h)$ and hence $\lambda=0$.
\end{proof}
\begin{pr}
	\label{proponondeg2}
	Let $\h$ be a $\la$-Einstein Lorentzian nilpotent Lie algebra with non-degenerate center.  Then $[\g,\g]_\g$ is a non-degenerate Lorentzian subspace of $\g$. Moreover, if $\h$ is $3$-step nilpotent and $\lambda\geq 0$ then $\mathrm{Z}(\g)=[\g,\g]_\g$.
\end{pr}
\begin{proof}
	According to \cite[Corollary 3.3]{bou0}, $[\h,\h]$ is nondegenerate Lorentzian and, one can easily see that $[\g,\g]_\g^\perp=[\h,\h]^\perp\cap\g$.
	Therefore $[\g,\g]_\g^\perp$ is nondegenerate Euclidean and hence $[\g,\g]_\g$ is nondegenerate Lorentzian. 
	
	Suppose now that $\h$ is $3$-step nilpotent. Then $\g$ is $2$-step nilpotent and therefore $[\g,\g]_\g\subset\mathrm{Z}(\g)$. Let $x\in\mathrm{Z}(\g)\cap[\g,\g]_\g^\perp$. Since $\ad_x=0$ and $J_x=0$, by virtue of \eqref{ricci1} ,  $\Ric_\g(x)=0$.  If $\lambda\geq 0$, the first equation of system \eqref{es} gives that :
	$$ 0\leq \lambda\langle x,x\rangle=-\dfrac12\mathrm{tr}(\omega_x^*\circ\omega_x)=Q.$$
	Since $\om$ is a 2-cocycle, $\om(Z(\G),[\G,\G]_\g)=0$ and hence
	\[ Q=-\dfrac12\sum_{i=1}^m\langle \omega(x,f_i),\omega(x,f_i)\rangle\leq 0 \]where $\{f_1,\dots,f_m\}$ is an orthonormal basis of $[\g,\g]_\g^\perp$. It follows that  $x\in\mathrm{Z}(\g)\cap\ker\om$ and hence $x=0$ by virtue of \eqref{rigid}. Thus $\mathrm{Z}(\g)=[\g,\g]_\g$.
\end{proof}

\begin{theo}\label{main1}
	Let $\h$ be a $\la$-Einstein Lorentzian $3$-step nilpotent  Lie algebra with nondegenerate center.  Then $\lambda\geq 0$.
\end{theo}

\begin{proof} According to \eqref{es}, since $\h$ is $\la$-Einstein then
	\begin{equation}\label{prel} \Ric_\G=\la \mathrm{Id}_\G+\frac12D\esp \tr(S_x\circ S_y)=-4\la \langle x,y\rangle_z, \end{equation}for any $x,y\in Z(\h)$. On the other hand, by virtue of Proposition \ref{proponondeg2}, $[\G,\G]$ is nondegenerate Lorentzian and hence $\G=[\G,\G]\oplus[\G,\G]^\perp$. We choose an orthonormal basis $\B_0=(e_1,\ldots,e_s)$ of $[\G,\G]$ with $\langle e_1,e_1\rangle_\G=-1$ and an orthonormal basis $\B_1=(z_1,\ldots,z_p)$ of $Z(\h)$ and we consider the Lie structure endomorphisms $(J_1,\ldots,J_s)$ associated to $\B_0$ and given by $\eqref{bracket}$ and $(S_1,\ldots,S_p)$ the $\om$-structure endomorphisms associated to $\B_1$ and given by \eqref{car}. 
	
	Since $\G$ is 2-step nilpotent then $[\G,\G]\subset Z(\G)$, hence for any $i=1,\ldots,s$, $J_i([\G,\G])=0$. Moreover, $J_i$ being skew-symmetric leaves $[\G,\G]^\perp$ invariant and we shall denote its restriction to $[\G,\G]^\perp$ by $J_i$ as well. On the other hand, since $\om$ is a 2-cocycle then $\om(Z(\G),[\G,\G])=0$ and hence, by virtue of \eqref{car}, for any $i=1,\ldots,p$, $S_i([\G,\G])\subset[\G,\G]^\perp$, we denote $B_i:[\G,\G]\too[\G,\G]^\perp$ the resulting linear map. Since $S_i$ is skew-symmetric, then for any $u\in[\G,\G]^\perp$, $S_iu=-B_i^*u+D_iu$ where $D_i:[\G,\G]^\perp\too[\G,\G]^\perp$ is skew-symmetric. By using \eqref{riccinilpotent}, \eqref{invariant} and \eqref{Df}, we get that \eqref{prel} is equivalent to
	\begin{equation}
	\label{3stepeq6}
	\begin{cases}
	-\dfrac{1}{2}J_1^2+\dfrac{1}{2}\mathlarger\sum\limits_{i=2}^s J_i^2+\dfrac{1}{2}\mathlarger\sum\limits_{i=1}^p (D_i^2-B_iB_i^*) = \lambda \mathrm{Id}_{[\g,\g]^\perp}.   \\
	\mathlarger\sum\limits_{i,j=1}^s \langle e_i,\;.\;\rangle\mathrm{tr}(J_i\circ J_j)e_j+2\mathlarger\sum\limits_{i=1}^p B_i^*B_i = -4\lambda \mathrm{Id}_{[\g,\g]}.   \\
	\mathrm{tr}(D_iD_j)-2\mathrm{tr}(B_i^*B_j)= -4\lambda\delta_{ij},\;i,j=1,\ldots,p.
	\end{cases}
	\end{equation}
	By taking the trace of the first two equations and  using the third one we obtain that :
	\begin{equation*}
	\mathlarger\sum_{i=1}^p\mathrm{tr}(D_i^2)=-4(2s+m+3p)\lambda,\quad m=\dim[\G,\G]^\perp.
	\end{equation*}
	But $[\G,\G]^\perp$ is a Euclidean vector space and $D_i:[\G,\G]^\perp\too[\G,\G]^\perp$ is skew-symmetric and hence $\tr(D_i^2)\leq0$ which completes the proof.
\end{proof}

To sum up the results of this section, we  reduced the study of Einstein Lorentzian $k$-step nilpotent Lie algebras to the study of a class of Lorentzian $(k-1)$-step nilpotent Lie algebras endowed with a 2-cocycle with values in a Euclidean vector space which in some cases can be Ricci-soliton. It is natural to give a name to this class of Lie algebras.

\begin{Def} \label{DefOmegaQuasi} A pseudo-Euclidean Lie algebra $(\G,\br_\G,\prs_\G)$ will be called $\om$-quasi Einstein of type $p$ if there exists $\la\in\R$ and a 2-cocycle $\om$ with values in a Euclidean vector space $(V,\prs_z)$ of dimension $p$ such that $\ker\om\cap Z(\G)=\{0\}$ and
	\[ \Ric_\G=\la\mathrm{Id}_\G+\frac12D,\;\;\;\;\tr(S_x\circ S_y)=-4\lambda\langle x,y\rangle_z \] where $S_x:\g\longrightarrow\g$ denotes the $\omega$-structure endomorphism corresponding to $x\in V$ and $D$ is given by
	\[ \langle Du,v\rangle_\G=\tr(\om_u^*\circ\om_v) \]and $\om_u:\G\too V$, $v\mapsto\om(u,v)$.
	
\end{Def}

\section{ $\om$-quasi Einstein Lorentzian 2-step nilpotent Lie algebras of type 1 }\label{section4}
In this section, having in mind Proposition \ref{proponondeg2} and Theorem \ref{main1}, we give a complete description of $\om$-quasi Einstein Lorentzian 2-step nilpotent Lie algebras of type 1 with nondegenerate Lorentzian derived ideal and Einstein constant $\la\geq0$ as an important step towards the determination of Einstein Lorentzian 3-step nilpotent Lie algebras with nondegenerate $1$-dimensional center.

Let $(\G,\br_\G,\prs_\G)$ be a 2-step nilpotent Lie algebra  such that $Z(\G)=[\G,\G]$ is nondegenerate Lorentzian.  Put $n=\dim[\G,\G]$ and $m=\dim[\G,\G]^\perp$.

Suppose that $\G$ is $\om$-quasi Einstein of type 1 with Einstein constant $\la\geq0$.
Denote by $S:\G\too\G$ the skew-symmetric endomorphism given by
$\om(u,v)=\langle Su,v\rangle_\G$. Since $\om$ is a 2-cocycle and $[\G,\G]\subset Z(\G)$ then $S([\G,\G])\subset[\G,\G]^\perp$ leading to a linear map $B:[\G,\G]\too[\G,\G]^\perp$.  The condition $Z(\G)\cap\ker\om=\{0\}$ implies that $B$ is injective. On the other hand, the skew-symmetry of $S$ gives that, for any $u\in[\G,\G]^\perp$, $Su=-B^*u+Lu$ where $L$ is a skew-symmetric endomorphism of $[\G,\G]^\perp$. Now consider the endomorphism $D$ associated to $\om$ and given by \eqref{D}. According to \eqref{Df}, $D=-S^2$ and hence
\[ Du=\begin{cases}B^*Bu-LBu\;\mbox{if}\; u\in[\G,\G],\\B^*Lu+BB^*u-L^2u
\;\mbox{if}\; u\in[\G,\G]^\perp.
\end{cases} \]
The fact that $\G$ is $\om$-quasi Einstein is equivalent to
\begin{equation}\label{meq}
-\frac12\mathcal{J}_1+\frac14\mathcal{J}_2-\frac12D=\la\mathrm{Id}_\G,\;\;\;\;\tr(S^2)=-4\lambda,
\end{equation} where, by virtue of \eqref{riccinilpotent}, $\Ric_\G=-\frac12\mathcal{J}_1+\frac14\mathcal{J}_2$.

Let us proceed now to a crucial step which is not possible to perform when $\om$ has it values in a vector space of dimension $\geq2$.

We consider  the symmetric endomorphism on $[\G,\G]$ given by $A=B^*B$. 
Since $B$ is injective and $[\G,\G]^\perp$ is nondegenerate Euclidean,
we have $\langle Au,u\rangle_\G>0$ for any $u\in\G\setminus\{0\}$. There are two types of nondiagonalizable symmetric endomorphisms on a Lorentzian vector space  (see \cite[p. 261-262]{Oneill}). Those which have an isotropic eigenvector or those which have two linearly orthogonal vectors $(e,f)$ such that $\langle e,e\rangle=1$, $\langle f,f\rangle=-1$,  $T(e)=ae-bf$ and $T(f)=be+af$. The fact that $A$ is  positive definite prevents it to be of these types and hence $A$ is diagonalizable in an orthonormal basis $\B_1=(e_1,\ldots,e_n)$ of $[\g,\g]$ such that $\langle e_1,e_1\rangle_\G=-1$.
Let $(J_1,\ldots,J_n)$ be the Lie structure endomorphisms associated to $\B_1$. Note that the $J_i$ vanishes on $[\G,\G]\subset Z(\G)$ and hence leaves invariant $[\G,\G]^\perp$. We denote the restriction of $J_i$ to $[\G,\G]^\perp$ by $J_i$ as well. Using \eqref{riccinilpotent} and \eqref{invariant}, we get that \eqref{meq} is equivalent to
\begin{equation}\label{meq1} \begin{cases}\di -\frac12J_1^2+\frac12\sum_{j=2}^nJ_j^2+\frac12(L^2-BB^*)=\la\mathrm{Id}_{[\G,\G]^\perp},\\\di
-2B^*B-\sum_{i,j=1}^n\langle e_i,u\rangle{\tr}(J_i\circ
J_j)e_j=4\la\mathrm{Id}_{[\G,\G]},\\\di
\tr(L^2)-2\tr(BB^*)=-4\lambda,\\[.1in]\di
LB=0.\end{cases} \end{equation}
Taking the trace of the first two equations and using the the third equation of \eqref{meq1} we get that :
$$ \tr(L^2)=-4(2n+m+3)\lambda,\;\;\;\;n=\dim[\g,\g],\;m=\dim[\g,\g]^\perp.$$
When $m=n$, $B:[\g,\g]\longrightarrow[\g,\g]^\perp$ is an isomorphism and therefore $LB=0$ leads to $L=0$ and by the previous equation $\lambda=0$. We will show that this fact is still true in the general setting.\\
Put $\B_2=(f_1,\ldots,f_n)=\left(\frac{B(e_1)}{|B(e_1)|},\ldots,\frac{B(e_n)}{|B(e_n)|}\right)$ which is obviously an orthonormal basis of $\img(B)$. Since $LB=0$,
$L$ vanishes on $\mathrm{Im}(B)$ and leaves invariant $\mathrm{Im}(B)^\perp=\ker BB^*$. Thus $L(f_i)=0$ and there exists an orthonormal basis $\B_3=(g_1,h_1,\ldots,g_r,h_r,p_1,\ldots,p_s)$ of $\ker BB^*$ such that
\[ L(g_i)=\mu_i h_i,\; L(h_i)=-\mu_i g_i, L(p_j)=0. \]
The basis $\B_1$ consists of eigenvectors of $B^*B$ and hence the second relation in \eqref{meq1} is equivalent to
\[ B^*B(e_i)=-\left(2\la+\frac12\langle e_i,e_i\rangle_\G\tr(J_i^2)\right)e_i,\;\tr(J_i\circ J_j)=0,\; i,j=1,\ldots,n, j\not=i. \]On the other hand, we also have, 
\begin{equation}\label{f} BB^*(f_i)=-\left(2\la+\frac12\langle e_i,e_i\rangle_\G\tr(J_i^2)\right)f_i,\quad i=1,\ldots,n. \end{equation}
Summing up the above remarks, if $M_i$ denotes the matrix of the restriction of $J_i$ to $[\G,\G]^\perp$ in the basis $\B_2\cup\B_3$ then \eqref{meq} implies that
\begin{equation}\label{M}
M_1^2-\sum_{k=2}^nM_k^2=\mathrm{Diag}\left(-\frac12\tr(M_1^2),\frac12\tr(M_2^2),\ldots,\frac12\tr(M_n^2),-(2\la+\mu_1^2),\ldots,-(2\la+\mu_r^2),-2\la,\ldots,-2\la   \right).
\end{equation}
To study this equation, we need matrix analysis of Hermitian square matrices (see \cite{horn2012matrix}). 
Let us recall one of the main theorems of this theory. A $m\times m$ Hermitian  matrix $A$ has real eigenvalues which can be ordered 
\[ \la_1(A)\leq\ldots\leq\la_m(A). \]
\begin{theo}[\cite{horn2012matrix}]
	\label{genth1}
	Let $A,B\in\mathcal{M}_m(\C)$ be two Hermitian matrices. Then for all $1\leq k \leq m$ :
	$$\lambda_k(A)+\lambda_1(B)\leq \lambda_k(A+B)\leq \lambda_k(A)+\lambda_m(B).$$
\end{theo}
Based on this theorem, the following lemma is a breakthrough in our study.
\begin{Le}
	\label{genlem0}
	Let $M_1,\dots,M_n$ be a family of skew-symmetric  $m\times m$ matrices with  $2\leq n\leq m$ and let $(v_1,\ldots,v_{m-n})$ be a family of nonpositive real numbers such that :
	\begin{equation}
	\label{geneq1}
	M_1^2-\sum_{l=2}^n M_l^2=\mathrm{Diag}\left(-\frac12\tr(M_1^2),\frac12\tr(M_2^2),\dots,\frac12\tr(M_n^2),v_1,\dots,v_{m-n}\right).
	\end{equation}
	Then $$(v_1,\ldots,v_{m-n})=(0,\ldots,0),  \; \lambda_1\left(\mathlarger\sum\limits_{l=2}^n M_l^2\right)=\mathlarger\sum\limits_{l=2}^n \lambda_1(M_l^2).$$ Moreover, for any $i\in\{2,\ldots,n\}$, $\mathrm{rank}(M_i)\leq2$.
\end{Le}
\begin{proof} Denote by $M$ the right-hand side of equation \eqref{geneq1}. By taking the trace of  \eqref{geneq1} we get :
	\begin{equation}
	\label{geneq4}
	\mathrm{tr}({M}_1^2)-\sum_{l=2}^n\mathrm{tr}({M}_l^2)=\frac23\sum_{i=1}^{m-n} v_i\leq0.
	\end{equation}
	For $i=1,\ldots,n$, $M_i^2$ is the square of a skew-symmetric matrix so its eigenvalues are real nonpositive and satisfies
	\begin{equation}\label{lam}\lambda_{2k-1}({M}_i^2)=\lambda_{2k}({M}_i^2),\quad  k\in\left\{1,\ldots,\left[\frac{m}{2}\right]\right\}.\end{equation}  It is clear that $-\frac{1}{2}\mathrm{tr}({M}_1^2)$ is the only nonnegative eigenvalue of $M$ and therefore  $\lambda_m(M)=-\frac{1}{2}\mathrm{tr}({M}_1^2)$.
	Theorem \ref{genth1} applied to \eqref{geneq1} gives that :
	\begin{equation}
	\label{geneq7}
	\underbrace{\lambda_m(M)+\lambda_1\left(\sum_{l=2}^n {M}_l^2\right)}_a\leq \lambda_m({M}_1^2)\leq \underbrace{\lambda_m(M)+\lambda_m\left(\sum_{l=2}^n {M}_l^2\right)}_b.
	\end{equation}
	and 
	\begin{equation}
	\label{geneq8}
	\underbrace{\lambda_{m-1}(M)+\lambda_1\left(\sum_{l=2}^n {M}_l^2\right)}_c\leq \lambda_{m-1}({M}_1^2)\leq \underbrace{\lambda_{m-1}(M)+\lambda_m\left(\sum_{l=2}^n {M}_l^2\right)}_d.
	\end{equation}
	Suppose that $m$ is odd. In this case  $\lambda_m({M}_1^2)=0$ and, by applying Theorem  \ref{genth1} inductively and using \eqref{lam}, we get that :
	\[ \frac{1}{2}\sum_{l=2}^n \mathrm{tr}({M}_l^2)\leq \frac{1}{2}\sum_{l=2}^n(\lambda_1({M}_l^2)+\lambda_2({M}_l^2))=\sum_{l=2}^n\lambda_1({M}_l^2)\leq \lambda_1\left(\sum_{l=2}^n {M}_l^2\right).\]
	As a consequence of this inequality and the fact that $\la_m(M)=-\frac{1}{2}\mathrm{tr}({M}_1^2)$, we get
	$$-\frac{1}{2}\mathrm{tr}({M}_1^2)+\frac{1}{2}\sum_{l=2}^n\mathrm{tr}({M}_l^2)\leq \lambda_m(M)+\lambda_1\left(\sum_{l=2}^n {M}_l^2\right)\stackrel{\eqref{geneq7}}\leq \la_m(M_1^2)\leq0,$$
	This combined with \eqref{geneq4} gives that $(v_1,\ldots,v_{m-n})=(0,\ldots,0)$. Suppose now that  $m$ is even. In this case,
	$\lambda_{m-1}({M}_1^2)=\lambda_m({M}_1^2)$ and it follows from \eqref{geneq7} and \eqref{geneq8} that $[a,b]\cap[c,d]\neq\emptyset$. But, we have obviously  that $c\leq a$ and $d\leq b$ therefore $a\leq d$. Thus
	\begin{equation}
	\label{geneq5}
	\lambda_m(M)+\lambda_1\left(\sum_{l=2}^n {M}_l^2\right) \leq \lambda_{m-1}(M)+\lambda_m\left(\sum_{l=2}^n {M}_l^2\right).
	\end{equation}
	Since $\la_m(M)=-\frac12\tr(M_1^2)$ then by using \eqref{geneq4} we get that
	\[ \lambda_m(M)+\lambda_1\left(\sum_{l=2}^n {M}_l^2\right)=-\frac{1}{2}\sum_{l=2}^n\mathrm{tr}({M}_l^2)-\frac13\sum_{i=1}^{m-n} v_i+\lambda_1\left(\sum_{l=2}^n {M}_l^2\right).  \]
	On other hand, by applying	Theorem \ref{genth1} once more, we get $\lambda_m(\sum_{l=2}^n {M}_l^2)\leq \sum_{l=2}^n \lambda_m({M}_l^2) \leq 0$, moreover  $\lambda_{m-1}(M)\leq 0$,  so \eqref{geneq5} implies that :
	\begin{equation}
	\label{eqeigen1}
	-\frac{1}{2}\sum_{l=2}^n\mathrm{tr}({M}_l^2)-\frac13\sum_{i=1}^{m-n} v_i+\lambda_1\left(\sum_{l=2}^n {M}_l^2\right)\leq 0,
	\end{equation}
	Theorem \ref{genth1} also implies that  $\lambda_1\left(\mathlarger\sum\limits_{l=2}^n {M}_l^2\right)\geq \mathlarger\sum\limits_{l=2}^n \lambda_1({M}_l^2)$ and hence
	\begin{eqnarray*}
		-\frac{1}{2}\sum_{l=2}^n\mathrm{tr}({M}_l^2)-\frac13\sum_{i=1}^{m-n} v_i+\lambda_1\left(\sum_{l=2}^n {M}_l^2\right) & \geq & -\frac{1}{2}\sum_{l=2}^n\mathrm{tr}({M}_l^2)-\frac13\sum_{i=1}^{m-n} v_i+\sum_{l=2}^n\lambda_1\left( {M}_l^2\right)\\
		& \geq & -\frac{1}{2}\sum_{l=2}^n\sum_{k=1}^m\lambda_k({M}_l^2)-\frac13\sum_{i=1}^{m-n} v_i+\sum_{l=2}^n\lambda_1( {M}_l^2)\\
		&\stackrel{\eqref{lam}} \geq & -\sum_{l=2}^n\sum_{k=1}^{[\frac{m}{2}]}\lambda_{2k-1}({M}_l^2)-\frac13\sum_{i=1}^{m-n} v_i+\sum_{l=2}^n\lambda_1( {M}_l^2)\\
		& \geq & -\sum_{l=2}^n\sum_{k=2}^{[\frac{m}{2}]}\lambda_{2k-1}({M}_l^2)-\frac13\sum_{i=1}^{m-n} v_i \geq 0.
	\end{eqnarray*}
	
	Again we get that $v_i=0$ for all $1\leq i\leq m-n$.
	
	To conclude, without any assumption on $m$, equation \eqref{geneq7} gives 
	\begin{eqnarray*}
		0\geq\lambda_m(M)+\lambda_1\left(\sum_{l=2}^n {M}_l^2\right)&=&-\frac{1}{2}\mathrm{tr}({M}_1^2)+\lambda_1\left(\sum_{l=2}^n {M}_l^2\right)\\
		&=&-\frac{1}{2}\sum_{l=2}^n\mathrm{tr}({M}_l^2)+\lambda_1\left(\sum_{l=2}^n {M}_l^2\right)\\
		&=&\lambda_1\left(\sum_{l=2}^n {M}_l^2\right)-\frac{1}{2}\sum_{l=2}^n\sum_{k=1}^m \lambda_k({M}_l^2)\\
		&=&\lambda_1\left(\sum_{l=2}^n {M}_l^2\right)-\sum_{l=2}^n\lambda_1({M}_l^2)-\frac{1}{2}\sum_{l=2}^n\sum_{k=3}^m \lambda_k({M}_l^2)\geq 0\\
	\end{eqnarray*}
	which means that $\lambda_1(\sum_{l=2}^n {M}_l^2)=\sum_{l=2}^n\lambda_1({M}_l^2)$ and $\lambda_k({M}_l^2)=0$ for all $3\leq k\leq m$ and $2\leq l \leq n$. This completes the proof.
\end{proof}

If we apply this lemma to our study, we get that $\la=0$, $L=0$ and $(J_2,\ldots,J_n)$ have rank 2 and satisfy $\la_1(\sum_{i=2}^nJ_i^2)=\sum_{i=2}^n
\la_1(J_i^2)$. The following lemma will give us a precise description of the endomorphisms $(J_2,\ldots,J_n)$.

\begin{Le}\label{genlem1} Let $V$ be an $m$-dimensional Euclidean vector space and
	$K_1,\dots,K_n:V\longrightarrow V$ a family of skew-symmetric endomorphisms with $n<m$. Assume that $\mathrm{rank}(K_i)=2$, $\mathrm{tr}(K_i\circ K_j)=0$ for all $i\neq j$ and that
	$$\lambda_1\left(K\right)=\sum_{i=1}^n \lambda_1\left(K_i^2\right)\;\;\;\text{with}\;\;\;K:=\sum_{i=1}^n K_i^2.$$
	Then we can find an orthonormal basis $\{u_0,\dots,u_{n},v_1,\ldots,v_{m-n-1}\}$ such that for all $1\leq i,j \leq n$, $1\leq l\leq m-n-1$ :
	$$ K_i(u_0)=\alpha_i u_i,\;K_i(u_j)=-\de_{ij}\alpha_i u_0\;\;\;\text{and}\;\;\;K_i(v_l)=0.$$
	
\end{Le}
\begin{proof} Consider $E:=\ker(K-\la_1(K)\mathrm{Id}_V)$ and for $i=1,\ldots,n$, denote $E_i:=\mathrm{Im}(K_i)$. Note that $E_i$ is a plan and there exists a $\al_i\in\R\setminus\{0\}$ such that for any $u\in E_i$, $K_i^2(u)=-\al_i^2 u$ and $\la_1(K_i^2)=-\al_i^2$. 
We claim that $E\subset\bigcap_{i=1}^n E_i$.  Indeed, let $u\in E$ and for each $i=1,\ldots,n$ choose an orthonormal basis $(e_i,f_i)$ of $E_i$ and write 
	$$u=\langle u,e_i\rangle e_i+\langle u,f_i\rangle f_i+v_i\esp v_i\in E_i^\perp.$$  Since $\la_1(K)=-\al_1^2-\ldots-\al_n^2$, we get
	\[ -\sum_{i=1}^n\al_i^2\langle u,u\rangle=\langle K^2(u),u\rangle=
	\sum_{i=1}^n\langle K_i^2(u),u\rangle. \]	
	But $K_i^2(u)=-\al_i^2(\langle u,e_i\rangle e_i+\langle u,f_i\rangle f_i)$ and hence
	\[ \langle K_i^2(u),u\rangle=-\al_i^2\left(\langle u,e_i\rangle^2 +\langle u,f_i\rangle^2 \right). \]So
	\[ 0=\sum_{i=1}^n\al_i^2(\langle u,u\rangle-\langle u,e_i\rangle^2 -\langle u,f_i\rangle^2)=\sum_{i=1}^n\al_i^2\langle v_i,v_i\rangle=0. \]	
	Thus $v_1,\ldots,v_n=0$ and the claim follows.\\\\ Choose $u_0\in E$  such that $\langle u_0,u_0\rangle=1$. Then clearly $(u_0,K_i(u_0))$ is an orthogonal basis of $E_i$. Complete this basis to get an orthonormal basis $(u_0,u_i,f_1,\ldots,f_{m-2})$ of $V$ with $u_i=\frac{1}{|K_i(u_0)|}K_i(u_0)$.
	We have $K_i(f_k)=0$ for $k=1,\ldots,m-2$ and hence for $i,j\in\{1,\ldots,n\}$ with $i\not=j$
	\begin{align*}
	0&=\tr(K_i\circ K_j)\\
	&=-\langle K_j(u_0),K_i(u_0)\rangle -\langle K_j(u_i),K_i(u_i)\rangle\\
	&=-\langle K_j(u_0),K_i(u_0)\rangle+\frac{\al_i^2}{|K_i(u_0)|}\langle K_j(u_i),u_0\rangle\\
	&=-\left(1+\frac{\al_i^2}{|K_i(u_0)|^2}\right)\langle K_j(u_0),K_i(u_0)\rangle.
	\end{align*}So the family $(u_0,K_1(u_0),\ldots,K_n(u_0))$ is orthogonal, we orthonormalize it and complete it to get the desired basis.
\end{proof}

The relevance of the following lemma will appear later.
\begin{Le}
	\label{Lemmaimp}
	Consider the following system of matrix equations on $\R^{2k}$ :
	\begin{equation}
	\label{systempric}
	\left\{%
	\begin{array}{l r}
	K^2=P^{-1}A P+A& \\
	\alpha K =AP-P^{-1}A &
	\end{array}
	\right.
	\end{equation}
	where $K$ is an invertible skew-symmetric matrix, $P$ an orthogonal matrix, $A=\mathrm{diag}(-\alpha_1^2,\dots,-\alpha_{2k}^2)$ with $\alpha_i\neq 0$ and $\alpha=\pm\sqrt{\alpha_1^2+\dots+\alpha_{2k}^2}$. Then $k=1$, in which case we get that :
	\begin{equation}
	\label{eqLemmaimp}
	A=\left(\begin{array}{l r}
	-\alpha_1^2 & 0\\0 & -\alpha_2^2\end{array}\right) ,\;\;\;K=\left(\begin{array}{l r}
	0 & \epsilon\sqrt{\alpha_1^2+\alpha_2^2}\\-\epsilon\sqrt{\alpha_1^2+\alpha_2^2} & 0\end{array}\right)\;\;\;\text{and}\;\;\;P=\left(\begin{array}{l r}
	0 & \mp\epsilon\\\pm\epsilon & 0\end{array}\right),\;\;\;\epsilon=\pm 1.
	\end{equation}
	$$$$	
\end{Le}
\begin{proof}
	We prove the Lemma by contradiction and assume that $(K,A,P)$ is a solution of \eqref{systempric} and  $k>1$. To get a contradiction, we prove first  that $K^2$ and $A$ commute and hence $A$ and $P^{-1}AP$ commute as well.
	
	Let  $\la_1<\ldots<\la_r<0$ be  the different eigenvalues of $K^2$ and $E_1,\ldots, E_r$ the corresponding vector eigenspaces. Since $K$ is skew-symmetric invertible and $\tr(K^2)=-2\al^2$, we have 
	\begin{equation}\label{k} \R^{2k}=E_1\oplus\ldots\oplus E_r, \dim E_i=2p_i\esp 2\sum_{i=1}^rp_i\la_i=-2\al^2. \end{equation}
	According to \eqref{systempric}, $P^{-1}AP+A$ and $AP-P^{-1}A$ commutes and hence
	\[ A(P+P^{-1})A=P^{-1}A(P+P^{-1})AP. \]
	Moreover the first equation of system \eqref{systempric} implies that 
	$$K^4=P^{-1}A^2P+A^2+AP^{-1}AP+P^{-1}APA$$
	and the second equation of \eqref{systempric} along with the preceding remarks give that :
	\begin{eqnarray*}
		\alpha^2K^2&=&APAP+P^{-1}AP^{-1}A-A^2-P^{-1}A^2P\\
		&=&APAP+P^{-1}AP^{-1}A+AP^{-1}AP+P^{-1}APA-K^4\\
		&=&(AP+AP^{-1})AP+P^{-1}A(P^{-1}A+PA)-K^4\\
		&=&A(P+P^{-1})AP+P^{-1}A(P^{-1}+P)A-K^4\\
		&=&A(P+P^{-1})A(P+P^{-1})-K^4.
	\end{eqnarray*}
	Therefore we get that $K^2(K^2+\alpha^2\mathrm{Id})=A(P+P^{-1})A(P+P^{-1})$ which leads to :
	\begin{equation}
	\label{comm1}
	A^{-1}K^2(K^2+\alpha^2\mathrm{Id})=(P+P^{-1})A(P+P^{-1}).
	\end{equation}
	But $P^{-1}=P^t$ and  the endomorphism at the right hand side of the previous equality is symmetric. This implies that  $A^{-1}$ and therefore $A$ commutes with $K^2(K^2+\alpha^2\mathrm{Id})$.\\
	W show now that $A$ commutes with $K^2$. If $K^2$ is proportional to $\mathrm{Id}$ this is obviously true. Suppose that $K^2$ has at least two distinct eigenvalues, i.e.,  $r\geq2$.
	For any $i,j\in\{1,\ldots,r\}$ and for any $u\in E_i,v\in E_j$, we have
	\begin{align*}
	\langle AK^2(K^2+\alpha^2\mathrm{Id})(v),w\rangle&=\la_i(\la_i+\al^2)\langle Au,v\rangle\\
	&=\langle K^2(K^2+\alpha^2\mathrm{Id})A(v),w\rangle\\
	&=\langle K^2(K^2+\alpha^2\mathrm{Id})w,A(v)\rangle\\
	&=\la_j(\la_j+\al^2)\langle Au,v\rangle.\end{align*}
	Thus, 
	\[ (\la_i-\la_j)(\la_i+\la_j+\al^2)\langle Au,v\rangle=0. \]
	But from \eqref{k}, we get
	\[ 2(\la_i+\la_j+\al^2)=-2(p_i-1)\la_i-2(p_j-1)\la_j-2\sum_{l\neq i,l\neq j}p_l\la_l\leq0. \]
	If $k>2$ then the last two   relations implies that if $i\not=j$ then $\langle A(E_i),E_j\rangle=0$ and hence $A(E_i)=E_i$ for $i=1,\ldots,r$. So $A$ commutes with $K^2$.
	
	If $k=2$ then $r=2$,  $\dim E_1=\dim E_2=2$ and $\la_1+\la_2=-\al^2$. From $\R^{2k}=E_1\oplus E_2$ one can deduce easily  that $K^2(K^2+\al^2\mathrm{Id})=-\la_1\la_2\mathrm{Id}$ and by replacing in \eqref{comm1} we get 
	$$A(P+P^{-1})=-\la_1\la_2(P+P^{-1})^{-1}A^{-1}.$$
	Now for any $u\in\R^{2k}$ we get that :
	$$ 0\geq \langle A(P+P^{-1})(u),(P+P^{-1})(u)\rangle=-\la_1\la_2\langle (P+P^{-1})^{-1}A^{-1}(u),(P+P^{-1})(u)\rangle=-\la_1\la_2\langle A^{-1}(u),u\rangle\geq 0,$$
	so $\langle A^{-1}(u),u\rangle=0$ which is impossible.

	In conclusion  $A$ commutes with $K^2$ and hence $A$ commutes with $P^{-1}AP$ so  
	there exists an orthonormal basis $\{v_1,\dots,v_{2k}\}$  of $\R^{2k}$ in which  both $A$ and $P^{-1}AP$ are diagonal.
	 For any $i\in\{1,\ldots,2k \}$
	$$Av_i=-\alpha_i^2 v_i\esp P^{-1}AP(v_i)=-\alpha_{\sigma(i)}^2 v_i$$
	for some permutation $\sigma$ of $\{1,\dots,2k\}$. The second equation of  \eqref{systempric} gives that, for any $i\in\{1,\ldots,2k \}$, 
	$$ \alpha K(v_i)=AP(v_i)-P^{-1}A(v_i)=-\alpha_{\sigma(i)}^2P(v_i)+\alpha_i^2P^{-1}(v_i).$$
	Thus 
	\begin{equation}
	\label{eqphew}
	\alpha^2\langle K(v_i),K(v_i)\rangle=\alpha_{\sigma(i)}^4+\alpha_i^4-2\alpha_{\sigma(i)}^2\alpha_i^2\langle P^2(v_i),v_i\rangle.
	\end{equation}
	First assume that $\sigma(i)=i$ for some $i\in\{1,\dots,2k\}$. It follows  from the first equation of  \eqref{systempric}  that $-2\alpha_i^2$ is an eigenvalue of $K^2$ so it must have multiplicity greater than $2$ and since $k>1$ this leads to $\tr(K^2)<-4\alpha_i^2$. On the other hand, equation \eqref{eqphew} and the first equation of  \eqref{systempric} imply $$\alpha^2\langle K(v_i),K(v_i)\rangle=2\alpha_i^4(1-\langle P^2(v_i),v_i\rangle)\esp
	-\langle K(v_i),K(v_i)\rangle=\langle K^2(v_i),v_i\rangle=-2\alpha_i^2.$$
	Combining these equations we obtain that $\alpha^2=\alpha_i^2(1-\langle P^2(v_i),v_i\rangle)$ and the Cauchy-Schwarz inequality $|\langle P^2(v_i),v_i\rangle|\leq \lVert v_i\lVert\;\lVert P^2v_i\lVert=1$ implies that $0\leq 1-\langle P^2(v_i),v_i\rangle\leq 2$ which in turn gives that $0\leq \alpha^2\leq 2\alpha_i^2$. Finally using that $\tr(K^2)=-2\alpha^2$ we conclude that $-4\alpha_i^2\leq\tr(K^2)$, and we get a  contradiction.
	Thus $\sigma(i)\neq i$ for all $i=1,\dots,2k$.
	
	From $\di\sum_{i=1}^{2k}\langle K(v_i),K(v_i)\rangle=-\tr(K^2)=2\alpha^2$ and equation \eqref{eqphew} we get :
	$$2\alpha^4=2\sum_{i=1}^{2k}\alpha_i^4-2\sum_{i=1}^{2k}\alpha_{\sigma(i)}^2\alpha_i^2\langle P^2(v_i),v_i\rangle.$$
	Now 
	\begin{align*}
	\alpha^4-\di\sum_{i=1}^{2k}\alpha_i^4&=(\al_1^2+\ldots+\al_{2k}^2)^2-\di\sum_{i=1}^{2k}\alpha_i^4\\
	&=\di\sum_{i\neq j}\alpha_i^2\alpha_j^2\\&=\di\sum_{i=1}^{2k}\alpha_i^2\alpha_{\sigma(i)}^2+
	\sum_{j\neq i,j\neq\sigma(i)}\alpha_i^2\alpha_j^2.\end{align*} So we obtain that :
	$$ 0\leq \sum_{j\neq i,\sigma(i)}\alpha_i^2\alpha_j^2=-\sum_{i=1}^{2k}\alpha_i^2\alpha_{\sigma(i)}^2(\langle P^2(v_i),v_i\rangle+1)\leq 0,$$
	the right hand side of the previous equality is negative as a consequence of the Cauchy-Schwarz inequality $|\langle P^2(v_i),v_i\rangle|\leq \lVert v_i\lVert\;\lVert P^2v_i\lVert=1$ which implies that $0\leq \langle P^2(v_i),v_i\rangle+1\leq 2$.\\\\ Thus $\di\sum_{j\neq i,\sigma(i)}\alpha_i^2\alpha_j^2=0$, but this contradicts the fact that $A$ is invertible. We conclude that $k=1$ and in this case we can   put :
	$$A=\left(\begin{array}{l r}
	-\alpha_1^2 & 0\\0 & -\alpha_2^2\end{array}\right) ,\;\;\;K=\left(\begin{array}{l r}
	0 & \beta\\-\beta & 0\end{array}\right)\;\;\;\text{and}\;\;\;P=\left(\begin{array}{l r}
	\cos(\theta) & -\sin(\theta)\\\sin(\theta) & \cos(\theta)\end{array}\right).$$
	We get that system \eqref{systempric} is equivalent to :
	$$	\left\{%
	\begin{array}{l r}
	\beta^2-\alpha_1^2-(\alpha_1^2\cos^2(\theta)+\alpha_2^2\sin^2(\theta))=0 \\\\
	\beta^2-\alpha_2^2-(\alpha_1^2\sin^2(\theta)+\alpha_2^2\cos^2(\theta))=0 \\\\
	\cos\theta\sin\theta(\alpha_2^2-\alpha_1^2)=0\\
	\pm\beta\sqrt{\alpha_1^2+\alpha_2^2}-(\alpha_1^2+\alpha_2^2)\sin\theta=0.
	\end{array}
	\right.$$
	By summing over the first two equations in the previous system and replacing in the last equation we obtain that $\beta=\epsilon\sqrt{\alpha_1^2+\alpha_2^2}$, $\sin\theta=\pm\epsilon$ and $\cos\theta=0$ with $\epsilon=\pm 1$, which ends the proof.
\end{proof}

We are now in possession of all the necessary ingredients to characterize $\om$-quasi Einstein Lorentzian 2-step nilpotent Lie algebras of type 1 as a key step toward the proof of Theorem \ref{Thprin1}. 
\begin{theo}
	\label{mainth1}
	Let $(\g,[\;,\;],\metric)$ be a Lorentzian $2$-step nilpotent Lie algebra and assume that $\mathrm{Z}(\g)=[\g,\g]$ is non-degenerate Lorentzian and let $\omega\in\mathrm{Z}^2(\g,\R)$. Then $\g$ is $\omega$-quasi Einstein of type $1$ with positive Einstein constant $\lambda$ if and only if $\lambda=0$ and, up to an isomorphism, $(\g,[\;,\;],\metric,\om)$ has one of the following forms :
	\begin{enumerate}
		\item  $\dim\g=5$  and there exists an orthonormal basis $\{e_1,e_2,u_1,u_2,u_3\}$ of $\g$ with $\langle e_1,e_1\rangle=-1$ such that the non vanishing Lie brackets and $\om$-products are given by :
		\begin{equation}
		\label{bracketstruct1dim6}
		[u_1,u_2]=\alpha e_2,\;[u_2,u_3]=\pm\alpha e_1,\;\omega(e_2,u_3)=\e\alpha,\;\omega(e_1,u_1)=\mp\e\alpha,\;\;\;\alpha\neq 0,\e=\pm1.
		\end{equation}
		
		\item $\dim\g=6$  and  there exists an orthonormal basis $\{e_1,e_2,e_3,u_1,u_2,u_3\}$ of $\g$ such that $\langle e_1,e_1\rangle=-1$ and the non vanishing Lie brackets and $\om$-products are given by :
		\begin{equation}
		\label{bracketstruct}
		\begin{cases}
		[u_1,u_2]=\alpha_2 e_2,\;[u_1,u_3]=\alpha_3 e_3,\;[u_2,u_3]=\epsilon\alpha e_1,\;\\\omega(e_2,u_3)=\mp\epsilon\alpha_2,\;\omega(e_3,u_2)=\pm\epsilon\alpha_3,\;\omega(e_1,u_1)=\pm\alpha,\end{cases}
		\end{equation}
		where $\alpha_2,\alpha_3\neq 0$, $\epsilon=\pm 1$ and $\alpha=\sqrt{\alpha_2^2+\alpha_3^2}$.
	\end{enumerate}
\end{theo}
\begin{proof}
	We keep the notations from the beginning of section \ref{section4}. The endomorphisms $(J_2,\ldots,J_n)$ have been shown to satisfy the hypothesis of Lemma \ref{genlem1}, therefore we can find an orthonormal basis $(u_1,u_2,\ldots,u_n,v_1,\ldots,v_{m-n})$ of $[\G,\G]^\perp$ and $(\al_2,\ldots,\al_n)\in\R^n$  such that, for all $2\leq i,j \leq n$ and all $1\leq k\leq m-n$,
	\[ J_i(u_1)=\al_i u_i,\; J_i(u_j)=-\de_{ij}\al_i u_1,\;\al_i\not=0\esp J_i(v_k)=0. \]
	Put $J=\sum_{i=2}^n J_i^2$, it is clear that for all $2\leq i \leq n$, $1\leq k\leq m-n$,
	\begin{equation}\label{r} J(u_1)=-(\al_2^2+\ldots+\al_n^2)u_1,\; J(u_i)=-\al_i^2 u_i, J(v_k)=0,\;\tr(J_1^2)=-2(\al_2^2+\ldots+\al_n^2)\esp  \tr(J_i^2)=-2\al_i^2. \end{equation}
	 Consider $\B_2=(f_1,\ldots,f_n):=\left(\frac{B(e_1)}{|B(e_1)|},\ldots,\frac{B(e_n)}{|B(e_n)|}\right)$. By virtue of equation \eqref{M}, we get that for any $i=2,\ldots,n$ and any $v\in\{f_1,\ldots,f_n\}^\perp$
	 \begin{equation}
	 \label{whtvr}
	 J_1^2(f_1)=J(f_1)-\frac12\tr(J_1^2)f_1,\;J_1^2(f_i)=J(f_i)+\frac12\tr(J_i^2)f_i \esp J_1^2(v)-J(v)=0. 
	 \end{equation}
    Since $\la_1(J)=\frac12\tr(J_1^2)$,  
	we deduce that
	\[ \langle J_1(f_1),J_1(f_1)\rangle=-\langle J(f_1),f_1\rangle+\la_1(J)\leq0  \]and hence 
	\[ J_1(f_1)=0\esp J(f_1)=\la_1(J)f_1. \]But \eqref{r} shows that the multiplicity of $\la_1(J)$ is equal to one and hence $f_1=\pm u_1$. Let us show that the restriction of $J_1$ to $f_1^\perp$ is invertible. We have from \eqref{meq1} that
	\[ J_1^2=J-BB^* \]and from \eqref{f} the restriction of $BB^*$ to $f_1^\perp$ is positive so if $u\in f_1^\perp$ and $J_1u=0$ we get
	\[ \sum_{i=2}^n\langle J_iu,J_iu\rangle+\langle BB^*(u),u\rangle=0 \]
	therefore $u\in\cap_{i=1}^n\ker J_i=Z(\G)=[\G,\G]$ and hence $u=0$. It follows that $J_1:f_1^\perp\too f_1^\perp$ is invertible and thus $m$ is odd. In view of the last equation of \eqref{whtvr} and the fact that $f_1=\pm u_1$, we obtain that $J_1^2(\{f_1,\dots,f_n\}^\perp)\subset\vect\{u_2,\dots,u_n\}$, the preceding remark then leads to $m-n\leq n-1$ thus $m\leq 2n-1$.\\\\
For convenience we set $w_i:=B(e_i)$ for $i=1,\ldots,n$. 	From \eqref{f} we get
	\[ \langle w_i,w_i\rangle=-\frac12\tr(J_i^2)\esp \langle w_i,w_j\rangle=0, i\not=j. \]

	So
	\begin{equation}
	\label{whtvr1}
	BB^*(x)=-(\al_2^2+\dots+\al_n^2)\langle x,u_1\rangle u_1+\sum_{i=2}^n\langle x,w_i\rangle w_i.
	\end{equation}
	The fact that $B$ defines a 2-cocycle is equivalent to
	\[ \sum_{i=1}^{n}\left( \langle J_i u,v\rangle w_i+\langle w_i,u\rangle J_i v-\langle w_i,v\rangle J_i u \right)=0,\quad u,v\in[\G,\G]^\perp. \]
	If we apply this  equation to $u=u_1$ we get
	\[ \langle w_1,u_1\rangle J_1 v=-\sum_{i=2}^{n}\left( \al_i\langle u_i ,v\rangle w_i-\al_i\langle w_i,v\rangle u_i  \right). \]
	From the definition of $w_1$ we get that $\langle w_1,u_1\rangle=\pm\sqrt{\alpha_2^2+\dots+\alpha_n^2}$ and therefore the previous equation gives that
	\begin{equation}
	\label{whtvr2}
	J_1=\pm\frac1{\sqrt{\al_2^2+\ldots+\al_n^2}}\sum_{i=2}^n\al_i u_i\wedge w_i.
	\end{equation}
	Actually this is equivalent to $B$ being a $2$-cocycle.
The expression of $BB^*$ given in \eqref{whtvr1} leads to
\begin{equation}
\label{whtvr3}
J_1^2-\sum_{i=2}^{q}J_i^2=(\al_2^2+\ldots+\al_n^2)\langle x,u_1\rangle u_1-\sum_{i=2}^n\langle .,w_i\rangle w_i.
\end{equation}
	Put $a=\pm\frac1{\sqrt{\al_2^2+\ldots+\al_n^2}}$. Equation \eqref{whtvr2} on the other hand gives that 
	\begin{align}
	\label{whtvr4}
	J_1w_l&=a\al_l^3 u_l-a\sum_{i=2}^n\al_i\langle u_i,w_l\rangle w_i,\nonumber\\
	J_1u_l&=-a\al_l w_l+a\sum_{i=2}^n\al_i\langle w_i,u_l\rangle u_i,\nonumber\\
	J_1v_k&=a\sum_{i=2}^n\al_i\langle w_i,v_k\rangle u_i,
	\end{align}
Now using \eqref{whtvr3} and then \eqref{whtvr4}, it is straightforward to check that
$$
		\langle J_1^2v_k,v_k\rangle=	-\sum_{l=2}^n\langle w_i,v_k\rangle^2
		=a\sum_{i=2}^n\al_i\langle w_i,v_k\rangle\langle J_1u_i,v_k\rangle
		=-a^2\sum_{l=2}^n\al_i^2\langle w_i,v_k\rangle^2.
$$	So we conclude that
	\[ \sum_{l=2}^n(1-a^2\al_i^2)\langle w_i,v_k\rangle^2=0. \]
	Thus either $n=2$ or $n\geq3$ and $\langle w_i,v_k\rangle=0$ for $i=1,\ldots,n$ and $v_k=1,\ldots,m-n$. So we get that either $n=2$ or $n\geq3$ and $m=n$.
	
	For $n=2$, we have $m=3$,  $(e_1,e_2)$ is an orthonormal basis of $[\G,\G]$ with $\langle e_1,e_1\rangle=-1$, $(u_1,u_2,v)$ an orthonormal basis of $[\G,\G]^\perp$,   $B(e_1)=a u_1$, $B(e_2)= bv$,
	\[ J_2=\left(\begin{array}{ccc}0&-\al&0\\\al&0&0\\0&0&0 \end{array}\right)\esp J_1=bu_2\wedge v=\left(\begin{array}{ccc}0&0&0\\0&0&b\\0&-b&0 \end{array}\right)\esp a^2=b^2=\al^2. \]
	This automatically leads to \eqref{bracketstruct1dim6}.
	For $n\geq3$, we have $n=m=2k+1$. Recall that
	$$[u,v]=\left\{\begin{array}{l c r}\mathlarger\sum\limits_{i=1}^n\langle J_i(u),v\rangle e_i,& &u,v\in[\g,\g]^\perp\\
	0,& &\text{otherwise}\end{array}\right.,\;\;\;\omega(u,v)=\left\{\begin{array}{l c r}\langle B(u),v\rangle,& &u\in[\g,\g],\;v\in[\g,\g]^\perp \\
	0,& &\text{otherwise}\end{array}\right.$$
	From what have been shown, the only Lie brackets of $\g$ that do not automatically vanish are 
	$$[u_1,u_i]=\langle J_i(u_1),u_i\rangle e_i=\alpha_i e_i\;\;\;\text{and}\;\;\;[u_i,u_j]=\langle J_1(u_i),u_j\rangle e_1:=\beta_{ij} e_1,$$
	for $2\leq i,j\leq n$, moreover since $J_1$ is invertible on $u_1^\perp$ it follows that $K:=(\beta_{ij})_{i,j}$ is a skew-symmetric invertible matrix. On the other hand, put $\hat{P}(f_i):=u_i$ for $2\leq i \leq m$ then $\hat{P}:=(\hat{p}_{ij})_{i,j}$ is an orthogonal matrix and it is straightforward to see that $\langle B(e_i),u_j\rangle=\langle B(e_i),\hat{P}(f_j)\rangle =\epsilon_i\hat{p}_{ji}\alpha_i$with $\epsilon_i=\pm 1$, it is clear that $P=(\epsilon_j\hat{p}_{ij})_{i,j}$ is an orthogonal matrix as well. Next since $f_1=\pm u_1$ we get that : 
	$$\langle B(e_1),u_1\rangle=\pm\sqrt{\alpha_2^2+\dots+\alpha_n^2}.$$ 
	Finally in these notations notice that $J_1^2-\sum_{i=2}^n J_i^2=-BB^*$ is equivalent to $K^2=P^{-1}AP+A$ with $A=\mathrm{diag}(-\alpha_2^2,\dots,-\alpha_n^2)$ and the cocycle condition $\oint\langle B([u,v]),w\rangle=0$ is equivalent to $\pm\alpha K=AP-P^{-1}A$ where $\alpha=\sqrt{\alpha_2^2+\dots+\alpha_n^2}$. This exactly the situation of Lemma \ref{Lemmaimp} and therefore $k=1$, i.e $n=m=3$ and thus $\dim\g=6$, furthermore in view of \eqref{eqLemmaimp} we get that the Lie algebra structure of $\g$ is given by \eqref{bracketstruct}. This ends the proof.
\end{proof}
Following the discussion of section \ref{section3} we get as a consequence of the preceding Theorem that a Lorentzian $3$-step nilpotent Lie algebras $(\h,\metric)$ with non-degenerate $1$-dimensional center is Einstein if and only if it is Ricci-flat and has one of the following forms :

	\begin{enumerate}
		\item Either $\dim\h=6$ in which case $\dim[\h,\h]=\mathrm{codim}[\h,\h]=3$ and there exists an orthonormal basis $\{x,e_1,e_2,u_1,u_2,u_3\}$ of $\h$ with $\langle e_1,e_1\rangle=-1$ such that the Lie algebra structure is given by :
		\begin{equation}
		\label{bracketstruct3dim6}
		[u_1,u_2]=\alpha e_2,\;[u_2,u_3]=\pm\alpha e_1,\;[e_2,u_3]=\alpha x,\;[e_1,u_1]=\mp\alpha x,\;\;\;\alpha\neq 0.
		\end{equation}
		\begin{equation}
		\label{bracketstruct4dim6}
		[u_1,u_2]=\alpha e_2,\;[u_2,u_3]=\pm\alpha e_1,\;[e_2,u_3]=-\alpha x,\;[e_1,u_1]=\pm\alpha x,\;\;\;\alpha\neq 0.
		\end{equation}
		\item $\dim\h=7$ in which case $\dim[\h,\h]=\mathrm{codim}[\h,\h]+1=4$. Moreover there exists an orthonormal basis $\{x,e_1,e_2,e_3,u_1,u_2,u_3\}$ of $\h$ such that $\langle e_1,e_1\rangle=-1$ and in which the Lie algebra structure is given by :
		\begin{equation}
		\label{bracketstruct3step}
		[u_1,u_2]=\alpha_2 e_2 ,\;[u_1,u_3]=\alpha_3 e_3,\;[u_2,u_3]=\epsilon\alpha e_1,\;[e_2,u_3]=\mp\epsilon\alpha_2x,\; [e_3,u_2]=\pm\epsilon\alpha_3x,\;[e_1,u_1]=\pm\alpha x
		\end{equation}
		where $\alpha=\sqrt{\alpha_2^2+\alpha_3^2}$.
	\end{enumerate}
\begin{proof}[Proof of Main Theorem.]In case $1$, the Lie algebra structure $[\;,\;]$ of $\h$ has one of the forms given by either \eqref{bracketstruct3dim6} or \eqref{bracketstruct4dim6}. It is clear that \eqref{bracketstruct4dim6} can be obtained from \eqref{bracketstruct3dim6} simply by replacing $u_3$ with $-u_3$, for this reason it suffices to treat the case where $\h$ is given by \eqref{bracketstruct3dim6}. Put :
	$$f_1=u_2,\;f_2=u_3+u_1,\;f_3=u_3-u_1,\;f_4=\pm \alpha e_1-\alpha e_2,\;f_5=\pm \alpha e_1+\alpha e_2,\;f_6=2\alpha^2x.$$
	Then we can easily see that :
	$$[f_1,f_2]=f_4,\;[f_1,f_3]=f_5,\;[f_2,f_4]=f_6,\;[f_3,f_5]=-f_6,$$
	$$[f_2,f_3]=[f_1,f_4]=[f_1,f_5]=[f_2,f_5]=[f_3,f_4]=[f_4,f_5]=[f_i,f_6]=0.$$
	Thus $\h\simeq\mathrm{L}_{6,19}(-1)$ and the metric $\metric$ is represented in the basis $\{f_1,\dots,f_6\}$ of $\h$ by the expression \eqref{6dimmetric1}. For case $2$, when $\h$ is given by \eqref{bracketstruct3step} we can put :
		$$f_1:=u_1,\;f_2:=u_2,f_3:=u_3,f_4:=\epsilon\sqrt{\alpha_2^2+\alpha_3^2}e_1,\;f_5:=\alpha_2 e_2,\;f_6=\alpha_3 e_3,\;f_7:=\pm\epsilon(\alpha_2^2+\alpha_3^2),$$
	then the Lie algebra structure of $\h$ is given by $\eqref{structdim7}$ with $r=\frac{\alpha_2^2}{\alpha_2^2+\alpha_3^2}$. Moreover if we set $a=\alpha_2^2+\alpha_3^2$ then we get that $\metric$ is given by \eqref{structmetric7}.
	\end{proof}
	We end our paper by some examples of Einstein Lorentzian nilpotent Lie algebras with non-degenerate center of dimension greater that one, the goal is to illustrate that such Lie algebras do occur even in the $3$-step nilpotent case. This gives motivation for a future investigation.
	
	\begin{exem}
		Let $\h$ be the $8$-dimensional nilpotent Lie algebra with Lie bracket $[\;,\;]$ given in a basis $\mathbb{B}=\{e_1,\dots,e_8\}$ by :
		$$\begin{cases}
		[e_1,e_2]=-4\sqrt{3}e_3,\;[e_1,e_3]=\sqrt{\dfrac{5}{2}}e_4,\;[e_1,e_4]=-2\sqrt{3}e_8,\;[e_1,e_5]=3\sqrt{\dfrac{7}{2}}e_6,\;\\
		[e_1,e_6]=-4\sqrt{2}e_7,\;[e_2,e_3]=-\sqrt{\dfrac{5}{2}}e_5,\;[e_2,e_4]=-3\sqrt{\dfrac{7}{2}}e_6,\;[e_2,e_5]=-2\sqrt{3}e_7,\vspace{.1in}\\
		[e_2,e_6]=-4\sqrt{2}e_8,\;[e_3,e_4]=-\sqrt{21}e_7,\;[e_3,e_5]=-\sqrt{21}e_8.
		\end{cases}$$
		One can define a Lorentzian inner product $\metric$ on $\h$ by requiring $\mathbb{B}$ to be an orthonormal basis with $\langle e_6,e_6\rangle=-1$. Then it is easy to see that $\mathrm{Z}(\h)=\vect\{e_7,e_8\}$ hence non-degenerate with respect to $\metric$. Moreover a straightforward computation shows that $(\h,\metric)$ is Einstein with nonvanishing scalar curvature. This example was first given in \cite{Dconti1}. 
	\end{exem}
	\begin{exem}
		Let $\metric$ be a Lorentzian metric on $\R^{7}$ and $\{e_1,\dots,e_{7}\}$ an orthonormal basis with respect to $\metric$ such that $\langle e_1,e_1\rangle=-1$. Define the Lie bracket $[\;,\;]$ by setting :	
		$$\begin{cases}
		[e_1,e_3]=\sqrt{2}e_7,\;[e_2,e_4]=\sqrt{2}e_7,\;[e_4,e_5]=-e_1,\;[e_4,e_6]=-e_1,\vspace{.1in}\\
		[e_3,e_5]=-e_2,\;[e_3,e_6]=-e_2.
		\end{cases}$$
		Put $\h:=(\R^{10},[\;,\;])$, then it is straightforward to check that $(\h,\metric)$ is a Ricci-flat $3$-step nilpotent Lie algebra with $\mathrm{Z}(\h)=\vect\{e_7,e_5-e_6\}$, therefore $\h$ has non-degenerate center.
	\end{exem}
	
	\begin{exem}
		Let $\metric$ be a Lorentzian metric on $\R^{10}$ and $\{e_1,\dots,e_{10}\}$ an orthonormal basis with respect to $\metric$ such that $\langle e_5,e_5\rangle=-1$. Choose $p,r\in\R$ such that $p,r\neq 0$ and define on $\R^{10}$ the Lie bracket $[\;,\;]$ given by :	
		$$\begin{cases}
		[e_1,e_3]=-\sqrt{p^2+r^2}e_5,\;[e_1,e_4]=-\sqrt{p^2+r^2}e_6,\;[e_2,e_4]=-\sqrt{p^2+r^2}e_5,\;[e_2,e_3]=-\sqrt{p^2+r^2}e_6,\vspace{.1in}\\
		[e_5,e_1]=pe_7,\;[e_5,e_2]=pe_8,\;[e_5,e_3]=re_9,\;[e_5,e_4]=re_{10}\vspace{.1in}\\
		[e_6,e_1]=pe_8,\;[e_6,e_2]=pe_7,\;[e_6,e_3]=re_{10},\;[e_6,e_4]=re_9.
		\end{cases}$$
		Put $\h:=(\R^{10},[\;,\;])$, then it is straightforward to check that $(\h,\metric)$ is a Ricci-flat $3$-step nilpotent Lie algebra with $\mathrm{Z}(\h)=\vect\{e_7,e_8,e_9,e_{10}\}$, therefore $\h$ has non-degenerate center.
	\end{exem}

\end{document}